\documentclass[12pt, reqno]{amsart}
\usepackage{amsmath, amsthm, amsfonts, amssymb, verbatim, graphicx, mathtools}
\usepackage{a4wide}
\usepackage{float}
\usepackage{enumerate}

\usepackage{tikz}
\usetikzlibrary{matrix}
\usetikzlibrary{backgrounds}

\usepackage{xcolor}

\usepackage{hyperref}

\usepackage{calc}

\usepackage{accents}

\allowdisplaybreaks

\newcommand{\R}{\mathbb{R}}

\newcommand{\N}{\mathbb{N}}

\newcommand{\Z}{\mathbb{Z}}

\newcommand{\sfa}{\mathsf{a}}
\newcommand{\sfk}{\mathsf{k}}

\definecolor{darkgreen}{rgb}{0.0, 0.65, 0.31}

\newtheorem{thm}{Theorem}[section]

\newtheorem{lem}[thm]{Lemma}
\newtheorem{prop}[thm]{Proposition}

{\theoremstyle{definition}

\newtheorem{rmk}[thm]{Remark}
\newtheorem{exmp}[thm]{Example}
}

\numberwithin{equation}{section}

\begin{document}
\title[]{A convergence criterion for the unstable manifolds of the MacKay approximate renormalisation}

\author{Seul Bee Lee}
\address{Centro di Ricerca Ennio De Giorgi, Scuola Normale Superiore, Piazza dei Cavalieri 3, 56126 Pisa, Italy}
\email{seulbee.lee@sns.it}

\author{Stefano Marmi}
\address{Scuola Normale Superiore, Piazza dei Cavalieri 7, 56126 Pisa, Italy}
\email{stefano.marmi@sns.it}

\author{Tanja I. Schindler}
\address{Faculty of Mathematics, University of Vienna, Oskar-Morgenstern-Platz 1, 1090 Vienna, Austria}
\email{tanja.schindler@univie.ac.it}

\date{\today}

\begin{abstract}
We give an explicit arithmetical condition which guarantees the existence of the unstable manifold of the MacKay approximate renormalisation scheme
for the breakup of invariant tori in one and a half degrees of freedom Hamiltonian systems,
correcting earlier results. 
Furthermore, when our condition is violated, we give an example of points on which the unstable manifold does not converge.
\end{abstract}


\thanks{The first and the third author acknowledge the support of the Centro di Ricerca Matematica
Ennio de Giorgi. All authors acknowledge the support of UniCredit Bank R\&D group for financial support through the ‘Dynamics and Information Theory Institute’ at the Scuola Normale Superiore.
The second and the third author acknowledge the support through the PRIN Grant "Regular and stochastic behaviour in dynamical systems" (PRIN 2017S35EHN)}

\maketitle

\section{Introduction and statement of result}

In \cite{mackay_exact_1988} MacKay introduced an approximate renormalisation scheme for determining the breakup threshold of invariant tori 
in one and a half degrees of freedom Hamiltonian systems of the Escande-Doveil \cite{escande1981renormalization} type. Mackay's renormalisation scheme is rather simple, involving only two operators, leading to explicit formulas for the stable and unstable manifolds
of the renormalisation fixed point for arbitrary
rotation numbers. Remarkably, 
it turns out that the points of convergence for these manifolds are connected to convergence properties of interesting arithmetical functions, for details see Remark \ref{rmk: comment on conditions}.

The aim of this note is to give convergence properties of the unstable manifold introduced in \cite{mackay_exact_1988}. The properties of this unstable manifold were already considered in \cite{stark_unstable_1989}, however it turns out that the convergence properties stated there are slightly wrong. 

In the following we first introduce the necessary notations from \cite{mackay_exact_1988}, to which we refer the reader for the general 
discussion of the renormalisation scheme. We sometimes use the typeface $\mathsf{A},\mathsf{B},\mathsf{C}$ in order to not confuse Mackay's notations with other later used notations. 
First we look at the operators $\mathsf{D}\colon (0,1)\times \R_{> 1}\times \R^2\to \R_{>0}^4$ and $\mathsf{S}\colon \R_{>1} \times \R_{>1}\times \R^2\to \R_{>0}^4$ introduced for a four-dimensional space and given by 
\begin{align*}
 \mathsf{D}(\omega, \sfk,\mathsf{a},A)&= \left(\frac{1}{\omega}-1, \frac{1}{\sfk}+1, A, A+\sfa+u(\sfk)\right),\\
 \mathsf{S}(\omega, \sfk,\sfa,A)&= \Big(\omega-1, \sfk+1,  \sfa, A+\sfa+u(\sfk)\Big)
\end{align*}
with 
\begin{align*}
 u(\sfk)=\log\left(\frac{1}{2\pi^2}\cdot \frac{(1+\sfk)^6}{\sfk^3}\right).
\end{align*}
We can unify these two operators by introducing 
\begin{align*}
 \mathsf{T}(\omega, \sfk,\sfa,A)&=\begin{cases}
                    \mathsf{D}(\omega, \sfk,\sfa,A)&\text{ if }\omega\in (0,1)\\
                    \mathsf{S}(\omega, \sfk,\sfa,A)&\text{ if }\omega>1
                   \end{cases}
\end{align*}
and set $(\omega_v, \sfk_v,\sfa_v,A_v)=\mathsf{T}^v(\omega, \sfk,\sfa,A)$. (In \cite{mackay_exact_1988} the successive application of $\mathsf{T}$ is denoted by a decoding $\mathsf{T}_{0}, \mathsf{T}_{1}, \ldots= \ldots \mathsf{T}_{1}\circ \mathsf{T}_{0}$.)

In order to state the expression for the unstable manifold 
of the renormalisation scheme in \cite{mackay_exact_1988}
we note that 
\begin{align*}
 \mathsf{T}^{-1}(\omega, \sfk, \mathsf{a}, A)=\begin{cases}
                                            \mathsf{S}^{-1}(\omega, \sfk, \mathsf{a}, A)&\text{if }\sfk>2\\
                                            \mathsf{D}^{-1}(\omega, \sfk, \mathsf{a}, A)&\text{if }\sfk<2,                                            
                                           \end{cases}
\end{align*}
set $n_0 =0$ and for $j\in\Z_{<0}$ we set
\begin{align}\label{eq:n_j negative}
 n_{j}=n_{j}(\sfk)=\max\left\{i<n_{j+1}\colon \left(\mathsf{T}^{-1}\circ \mathsf{T}^{i+1}\right)(\omega,\sfk,\sfa,A)= \left(\mathsf{D}^{-1}\circ \mathsf{T}^{i+1}\right)(\omega,\sfk,\sfa,A)\right\},
\end{align}
see \cite[p.~254]{mackay_exact_1988}.
Then the unstable manifold is given by 
\begin{align}
 A_0-\sfk\, \sfa_0-\sum_{J=0}^{-\infty}\left[\prod_{j=-1}^J( -\sfk_{n_j}^{-1})\sum_{n_{J-1}\leq v< n_J} u(\sfk_v) \right]=0\label{eq: def unstable}
\end{align}
with the convention $\prod_{j=-1}^0 (-\sfk_{n_j}^{-1})=1$.
In order to state our main proposition we first introduce some notation from the theory of continued fractions. 
Here and in the following we denote by
\begin{align*}
 x=a_0(x)+\cfrac{1}{a_1(x)+\cfrac{1}{a_2(x)+\ddots+\cfrac{1}{a_n(x)+\cfrac{1}{\kappa}}}}=[a_0(x); a_1(x),a_2(x),\ldots, a_n(x)+{\kappa^{-1}}]
\end{align*}
the standard continued fraction expansion of a positive number $x$ calculated up to its $n$th convergent where $\kappa>1$. If $\kappa$ equals a natural number, then $\kappa=a_{n+1}(x)$ 
and $x$ is rational and has a finite continued fractions expansion. 
Furthermore, let $p_n(x)/q_n(x)$ be the $n$th principal convergent of a positive number $x$, i.e.\ $p_n(x)/q_n(x)=[a_0(x);a_1(x),\ldots,a_n(x)]$. 
With this notation we are able to state our main proposition.
\begin{prop}\label{prop: main prop}
 The sum of the unstable manifold converges if and only if 
 \begin{equation}
 \lim_{j\to\infty}\frac{\log q_{j+1}\left(\sfk_{n_{-1}}^{-1}\right)}{q_j\left(\sfk_{n_{-1}}^{-1}\right)}=0.\label{eq: criterion}
\end{equation}
\end{prop}
 Before giving the proof of this proposition we will first give some remarks on this convergence property and the convergence property of the stable manifold.

 \begin{rmk}\label{rmk: comment on conditions}
Eq.\ \eqref{eq: criterion} is a criterion which has also applications in the theory of small divisor problems: 
indeed, it is the optimal arithmetic condition that guarantees that the cohomological equation associated with an irrational rotation of the circle has an analytical solution whatever the analytical datum given is: see, e.g.\ \cite[Prop.~A3.4]{marmi2000introduction}.

  In \cite{stark_unstable_1989} it was erroneously stated that the sum of the unstable manifold converges for each irrational number: the above proposition corrects this and in Example \ref{ex: main ex} we will give an explicit example of points for which the unstable manifold does not converge. 

We also want to emphasize that the convergence of the stable manifold adheres to a related but stronger condition. Namely, in \cite{mackay_erratum_1989} it is stated that the sum of the stable manifold (not given in these notes) converges if and only if $\omega_1$ is a Brjuno number, i.e.\ if 
$\sum_{j=1}^{\infty}\frac{\log q_{j+1}(\omega)}{q_j(\omega)}<\infty$.
\end{rmk}

Before giving the proof of the above proposition, we will first introduce $G(x)=1/x-\lfloor 1/x\rfloor$ and $\beta_j(x)=\prod_{i=0}^j G^i (x)$ 
for $j \geq 0$ with the convention $\beta_{-1}(x) = 1$ and give a lemma, describing some convergence properties in the language of continued fractions.
In this lemma, we omit the dependence on $x$, implicitly assuming dependence on the same number for each term. 
\begin{lem}\label{lem: criterion}
We have the following two statements:
\begin{enumerate}
  \item\label{en: 1a} 
 We have $\sum_{j=1}^{\infty}\beta_{j-1} a_j<\infty$.
 \item\label{en: 3} 
 Eq.\ \eqref{eq: criterion} is equivalent to each of the following two statements:
 \begin{enumerate}
  \item\label{en: crit1} $\lim_{j\to\infty}\beta_{j-1} \log a_{j+1}=0$,
  \item\label{en: crit2} $\lim_{j\to\infty}\beta_{j-1} a_{j} \log a_{j}=0$.
 \end{enumerate}
 \end{enumerate}
\end{lem}
\begin{proof}
To show \eqref{en: 1a}, we use $q_n=a_n\,q_{n-1}+q_{n-2}$ (see \cite[Eq.\ (1.17)]{MMY1}) and $\frac{1}{2}< q_{n+1}\,\beta_{n}<1$ (see \cite[Prop.~1.4 (iii)]{MMY1}) and obtain 
$\sum_{j=1}^{\infty}\beta_{j-1} a_j
 <\sum_{j=1}^{\infty} \frac{1}{q_{j-1}}<\infty$.

To show \eqref{en: 3},
we apply $q_n=a_n\,q_{n-1}+q_{n-2}$, $\frac{1}{2}< q_{n+1}\,\beta_{n}<1$,
and then two times $q_{j}\ge 2 q_{j-2}$ (which can be obtained from $q_n=a_n\,q_{n-1}+q_{n-2}$)
and obtain
\begin{align}
\MoveEqLeft\beta_{j-1}\cdot  a_{j}\cdot \log a_{j}\notag\\
&> \frac{1}{2}\cdot\frac{q_{j}- q_{j-2}}{q_{j} q_{j-1}}
\cdot \left(\log \left(q_{j}-q_{j-2}\right)-\log q_{j-1}\right)
>\frac{1}{4}\left[\frac{\log q_{j}}{q_{j-1}}
- \frac{\log q_{j-1}-\log\left(\frac{1}{2}\right)}{q_{j-1}}\right]
\label{eq: brjuno1}
\end{align}
and on the other hand 
$\beta_{j-1}\cdot  a_{j}\cdot \log a_{j}
 <
 \frac{\log q_{j}}{q_{j-1}}$.
Then the equivalence between \eqref{eq: criterion} and \eqref{en: crit2} follows by the last two equations. 
 Furthermore, by $\beta_{j-2} = a_{j} \beta_{j-1}+\beta_j$ (see \cite[Eq.~(1.22)]{MMY1}), we have 
 $\beta_{j-1} a_{j} \log a_{j}=\left(\beta_{j-2}-\beta_j\right)\log a_j$. Since  $\lim_{j\to\infty}\beta_j\log a_j<\lim_{j\to\infty} \log q_{j}/q_{j+1}=0$, we have $\lim_{j\to\infty}\beta_j\log a_j=0$ and we also obtain the equivalence between \eqref{en: crit1} and \eqref{en: crit2}. 
\end{proof}
 
\begin{proof}[Proof of Proposition \ref{prop: main prop}]
In \cite[p.\ 215]{stark_unstable_1989} it is claimed that in case $\sfk_0\in (1,2)$ we have
 \begin{align*}
  \sum_{J=0}^{-\infty}\left[\prod_{j=-1}^J( -\sfk_{n_j}^{-1})\sum_{n_{J-1}\leq v< n_J} u(\sfk_v) \right]
  &=\sum_{J=-1}^{-\infty}\left[\mathsf{\nu}\left(\sfk_{n_J}\right)\prod_{j=-1}^{J+1}( -\sfk_{n_j}^{-1}) \right]
 \end{align*}
 with $\mathsf{\nu}(k)=u(k)-U(k)/k$ and $U(k)=u(k-1)+\ldots +u(k-n)$ such that $k-n\in (1,2)$. 
 However, the proof in \cite{stark_unstable_1989} uses a reordering of an alternating sum which is only possible if the summands tend to zero. 
 If this is the case we can verify his claim as follows: 
 We have $\sum_{n_{J-1}< v< n_J} u(\sfk_v)=U(\sfk_{n_J})$ and thus, 
 \begin{align}
  \MoveEqLeft\sum_{J=0}^{-\infty}\left[\prod_{j=-1}^J( -\sfk_{n_j}^{-1})\sum_{n_{J-1}\leq v< n_J} u(\sfk_v) \right]
  =\sum_{J=0}^{-\infty}\left[\prod_{j=-1}^{J}( -\sfk_{n_j}^{-1})\left( u(\sfk_{n_{J-1}})+ U(\sfk_{n_{J}})\right) \right]\notag\\
  &=\sum_{J=0}^{-\infty}\left[\prod_{j=-1}^{J}( -\sfk_{n_j}^{-1})\, u(\sfk_{n_{J-1}})\right]+ \sum_{J=0}^{-\infty}\left[\prod_{j=-1}^{J}( -\sfk_{n_j}^{-1})\, U(\sfk_{n_{J}}) \right]\label{eq: two Wilton sums}\\
  &=\sum_{J=-1}^{-\infty}\left[\prod_{j=-1}^{J+1}( -\sfk_{n_j}^{-1})\, u(\sfk_{n_{J}})\right]
  + U(\sfk_{n_{0}})
  -\sum_{J=-1}^{-\infty}\left[\prod_{j=-1}^{J+1}( -\sfk_{n_j}^{-1})\, \frac{U(\sfk_{n_{J}})}{\sfk_{n_{J}}} \right]\notag\\
  &=U(\sfk_{n_0})+\sum_{J=-1}^{-\infty}\left[\prod_{j=-1}^{J+1}( -\sfk_{n_j}^{-1}) \nu\left(\sfk_{n_{J}}\right) \right]\notag
 \end{align}
and $U(\sfk_{n_0})=0$ if $\sfk_{n_0}\in (1,2)$. 
However, as mentioned before we have reordered the sum in the second equality. This reordering is only possible if 
\begin{equation}\label{eq: conv cond product}
 \lim_{J\to-\infty}\left(\prod_{j=-1}^{J}\sfk_{n_j}^{-1}\right) u(\sfk_{n_{J-1}})=0
 \quad\text{ and }\quad 
 \lim_{J\to-\infty}\left(\prod_{j=-1}^{J}\sfk_{n_j}^{-1}\right) U(\sfk_{n_{J}})=0.
\end{equation}
If this convergence holds the remaining part of the proof follows as in \cite{stark_unstable_1989}.

On the other hand, it also becomes immediately clear that the sum can only converge if \eqref{eq: conv cond product} holds. 
Hence, we are left to show that \eqref{eq: conv cond product} is equivalent to \eqref{eq: criterion}. 
If we set, $m_J:=n_{J+1}-n_J$, then by \eqref{eq:n_j negative} we have
\begin{align*}
 \mathsf{T}^{n_J}
 &=\mathsf{D}^{-1} \mathsf{S}^{-m_J+1}\mathsf{D}^{-1} \mathsf{S}^{-m_{J+1}+1}\ldots \mathsf{D}^{-1} \mathsf{S}^{-m_{-1}+1},
\end{align*}
for $J\leq 0$. This also implies 
\begin{equation}\label{eq: k0 CF}
 \sfk_0 = [m_{-1};m_{-2},m_{-3},\cdots].
\end{equation}

For the first term in \eqref{eq: conv cond product} we obtain by using $\sfk_{n_J}\in (m_{J-1}, m_{J-1}+1)$
and the definition of $u$ that 
$u(\sfk_{n_{J-1}})= 3\log m_{J-2}+o_{m_{J-2}}(1)$ 
where $o_{\log m_{J-2}}(1)$ is understood as $o(1)$ with respect to $m_{J-2}( \sfk_0)$ being large. 
Furthermore, we have by \eqref{eq: k0 CF} that $m_{J}= a_{-J}(\sfk_0^{-1})=a_{-J-1}(\sfk_{n_{-1}}^{-1})$, for $J<0$.
On the other hand, we note that 
$\pi_2\left(\mathsf{D}^{-1}\mathsf{S}^{-\lfloor \sfk-1\rfloor}\right)(\omega, \sfk, \mathsf{a}, A)
=(\sfk- \lfloor \sfk\rfloor)^{-1}$
and we have 
$\sfk_{n_{j-1}}^{-1}=1/ \sfk_{n_j}^{-1} - \lfloor 1/ \sfk_{n_j}^{-1}\rfloor= G\big(\sfk_{n_j}^{-1}\big)$
which implies 
\begin{equation}
 \prod_{j=-1}^J( -\sfk_{n_j}^{-1})=(-1)^J\, \beta_{-J-1}\left(\sfk_{n_{-1}}^{-1}\right).
\end{equation}
Hence, 
\begin{equation*}
 \lim_{J\to-\infty}\left(\prod_{j=-1}^{J}\left(-\sfk_{n_j}^{-1}\right)\right) u(\sfk_{n_{J-1}})
 =\lim_{J\to-\infty} 3\,   (-1)^J\, \beta_{-J-1}\left(\sfk_{n_{-1}}^{-1}\right)\, \log a_{-J+1}(\sfk_{n_{-1}}^{-1})
\end{equation*}
and by Lemma \ref{lem: criterion} \eqref{en: 3} the first term in \eqref{eq: conv cond product} converges if and only if \eqref{eq: criterion} holds. 

Next, we consider the second term in \eqref{eq: conv cond product}.
From the definition of $u$ we immediately get
\begin{equation}\label{eq: MacKay sum}
 U(\sfk_{n_J})=\sum_{n_{J-1}< v < n_{J}} u(\sfk_v) = (a_{-J}(\sfk_{n_{-1}}^{-1})-1) \log \frac{1}{2\pi^2} + 3 \log \prod_{n_{J-1}< v< n_J} \frac{(1+\sfk_v)^2}{\sfk_v}.
\end{equation}
Since $\sfk_{n_{J-1}+i} = 1/\sfk_{n_{J-1}} + i$ for $i=1,\cdots, m_{J-1}-1$, we have
\begin{align*}
 \prod_{n_{J-1}< v< n_{J}} \frac{(1+\sfk_v)^2}{\sfk_v}
&= (\sfk_{n_{J-1}}^{-1}+1)^{-1} (\sfk_{n_{J-1}}^{-1}+2)\, (\sfk_{n_{J-1}}^{-1}+3)\cdots \\
&\qquad \cdot (\sfk_{n_{J-1}}^{-1}+a_{-J}(\sfk_{n_{-1}}^{-1})-1)(\sfk_{n_{J-1}}^{-1}+{a_{-J}}(\sfk_{n_{-1}}^{-1}))^2
=: S_J.
\end{align*}
This implies 
$$\frac{1}{2}\cdot a_{-J}(\sfk_{n_{-1}}^{-1})\cdot \left(a_{-J} (\sfk_{n_{-1}}^{-1})\right) ! \le S_J \le   (a_{-J}(\sfk_{n_{-1}}^{-1})+1)\cdot (a_{-J}(\sfk_{n_{-1}}^{-1})+1)!.$$
Using Stirling's formula yields 
$\log S_J= a_{-J}(\sfk_{n_{-1}}^{-1})\,\log a_{-J}(\sfk_{n_{-1}}^{-1})+O_{a_{-J}}(a_{-J}(\sfk_{n_{-1}}^{-1}))$. 
Combining this with \eqref{eq: MacKay sum} yields 
\begin{align*}
 \MoveEqLeft\left(\prod_{j=-1}^{J}-\sfk_{n_j}^{-1}\right) U(\sfk_{n_{J}})\\
 &=(-1)^J\, \beta_{-J-1}\left(\sfk_{n_{-1}}^{-1}\right) \left( 3a_{-J}\left(\sfk_{n_{-1}}^{-1}\right)\cdot \log a_{-J}\left(\sfk_{n_{-1}}^{-1}\right) +O_{a_{-J}}\left(a_{ -J}\left(\sfk_{n_{-1}}^{-1}\right)\right)\right).
\end{align*}
By Lemma \ref{lem: criterion}-\eqref{en: 1a}, we have that the sum over $\beta_{-J-1}O_{a_{-J}}(a_{-J})$ and converges absolutely. 
Applying Lemma \ref{lem: criterion}-\eqref{en: 3}
yields that \eqref{eq: criterion} is a necessary condition for \eqref{eq: def unstable} to converge.
\end{proof}
 
\begin{exmp}\label{ex: main ex}
 Criterion \eqref{eq: criterion} is a relatively weak condition. However, it is not too difficult to construct a number $\sfk_{n_{-1}}^{-1}$ in terms of its continued fraction entries $(a_i)=\big(a_i(\sfk_{n_{-1}}^{-1})\big)_{i\in\N}$ such that \eqref{eq: criterion} does not hold. Take for example $a_1=1$ and $a_i=2^{a_{i-1}^2}$. 
 Then we clearly have $q_{j+1}\geq a_{j+1}$. On the other hand, we have 
 $q_j\leq \left(a_{j}+1\right)q_{j-1}\leq \prod_{i=1}^{j}\left(a_i+1\right)$. 
 Thus, using $a_{i-1}=\left(\log a_{i}/\log 2\right)^{1/2}\leq \log a_{i}$ we obtain
 \begin{align*}
  \frac{\log q_{j+1}}{q_j}
  &\geq \frac{ \log a_{j+1}}{\prod_{i=2}^{j}\left(a_i+1\right)}
  \geq \frac{ a_{j}^2 \log 2 }{2^j a_j (\log a_j) (\log\log a_j) \dots (\log\log\cdots \log a_j) }
  \geq a_{j}^{1/2}, 
 \end{align*}
where $\log$ is iterated $j-2$ times in the last factor of the denominator of the third term and the last inequality holding for $j$ being large. This term obviously does not converge to zero. 
 
 However, if we choose $a_1=1$ and $a_i=2^{a_{i-1}}$ instead, then $q_{j+1}\leq (a_{j+1}+1)q_j$ and $q_j\geq a_{j} q_{j-1}\geq \prod_{i=0}^{j} a_i$
  imply
 \begin{align*}
  \frac{\log q_{j+1}}{q_j}
  &\leq \frac{ \log (a_{j+1}+1)+\log q_j}{q_j}
  \leq \frac{ \sum_{i=1}^{j+1}\log (a_i+1) }{\prod_{i=1}^{j} a_i}
  \lesssim \frac{ \sum_{i=1}^{j}a_{i}\log 2}{\prod_{i=1}^{j} a_i},
 \end{align*}
which tends to zero for $j$ tending to infinity.
Given how fast at least some of the continued fraction entries have to grow such that \eqref{eq: criterion} is not fulfilled it is very difficult to detect those points numerically.
\end{exmp}

\section*{Acknowledgments}
This article was inspired by R.S.\ Mackay's observations after a Stony Brook University Renormalisation online seminar given by the second author on Brjuno and Wilton functions, see \cite{W, lee_brjuno_2021}. 
Indeed, it turns out that in \eqref{eq: two Wilton sums} there are two Wilton sums which cancel each other out. We are very grateful to R.S. Mackay for his suggestions and further discussions.

\end{document}